\documentclass[12pt]{amsart}
\usepackage{graphicx}
\usepackage{amssymb}
\usepackage{amsthm}
\usepackage{amsmath}
\usepackage{mathrsfs}
\usepackage{enumitem}
\usepackage{tikz}

\newcommand{\RR}{\mathbb{R}}
\newcommand{\KK}{\mathbf{K}}

\newcommand{\THETABODY}{\textup{TH}}

\newtheorem{theorem}{Theorem}[section]

\newtheorem{lemma}[theorem]{Lemma}
\theoremstyle{definition}

\theoremstyle{remark}

\theoremstyle{corollary}
\newtheorem{corollary}[theorem]{Corollary}
\theoremstyle{proposition}
\newtheorem{proposition}[theorem]{Proposition}

\title{A Semidefinite Approach to the $K_i$ Cover Problem}
\author{Jo{\~a}o Gouveia}
\address{Jo{\~a}o Gouveia, CMUC, Department of Mathematics,
  University of Coimbra, 3001-454 Coimbra, Portugal}
\email{jgouveia@mat.uc.pt} 
\author{James Pfeiffer}
\address{James Pfeiffer, Department of Mathematics,
University of Washington, Seattle, WA 98195}
\email{jamesrpfeiffer@gmail.com}
\thanks{The authors were partially supported on this project 
as follows: JG by `Centro de Matem\'{a}tica da Universidade de Coimbra' and `Funda\c{c}\~{a}o para a Ci\^{e}ncia e a Tecnologia', through European program COMPETE/FEDER; and JP by NSF grant DMS-1115293}

\begin{document}

\maketitle

\begin{abstract}
We apply theta body relaxations to the $K_i$-cover problem and show
polynomial time solvability for certain classes of graphs. In particular, we
give an effective relaxation where all $K_i$-$p$-hole facets are valid,
and study its relation to an open question of Conforti et al. For the triangle free problem, we show for $K_n$ that the theta body relaxations do not converge by $(n-2)/4$ steps; we also prove for all $G$ an integrality gap of 2 for the second theta body.
\end{abstract}

\section{Introduction}
A common way to model a combinatorial optimization problem is as the optimization of a function over the set $S \subseteq \{0,1\}^n$ of  characteristic vectors of the objects in question. When the objective function is linear, 
we may replace $S$ by its convex hull $\hbox{conv}(S)$. The problem can be solved efficiently if we can find a small description of this polytope. Since for NP hard problems we cannot expect this, we look instead for  approximations to $\hbox{conv}(S)$. One possibility is to use semidefinite approximations, as introduced by Lov\'{a}sz \cite{lovasz} with the construction of the {\em theta body} of the stable set polytope of a graph. Another famous example is the approximation algorithm for the max cut problem due to Goemans and Williamson \cite{goemans_williamson}. In this paper we will use the semidefinite relaxations introduced by Gouveia, Parrilo and Thomas \cite{gpt} to analyze the {\em $K_i$-cover problem}.

Recall that $K_i$ denotes the complete graph, or clique, on $i$ vertices. Given a graph $G$, let $\KK_j(G)$ be the collection of cliques in $G$ of size $j$ (usually, the graph is clear from context, and we write $\KK_j$). A (possibly empty) collection $C \subset \KK_{i-1}$ is said to be a $K_i$-cover if for each $K \in \KK_i$, there is some $H \in C$ with $H \subset K$. In this case we say that $H$ covers $K$. The $K_i$-cover problem is, given a graph $G$ and a set of weights on $\KK_{i-1}$, to compute the minimum weight $K_i$-cover. The case $i=2$ is more commonly known as the vertex cover problem, in which we seek a collection $C$ of vertices such that each edge in $G$ contains at least one vertex from $C$. However, note that the usage of ``cover'' is reversed here: the vertex cover problem is the $K_2$-cover problem, not the $K_1$-cover problem.

A closely related problem, and the setting in which we will prove our results, is the {\em $K_i$-free problem}. As before, we are given a graph and a collection of weights on $\KK_{i-1}$. But now we seek the maximum weight collection $C \subseteq \KK_{i-1}$ such that $C$ is $K_i$-free. That is, for each $K \in \KK_i$, there is some $H \in \KK_{i-1}$, with $H \subset K$ and $H \notin C$. Again, the case $i=2$ of this problem is well-known as the stable set problem: we seek a maximum weight {\em stable set} $C$, where $C$ is stable if no two of its vertices are connected by an edge.

The vertex cover and stable set problems are related in the following sense: let $G = (V,E)$ be a graph. Then a subset $C$ of vertices is a vertex cover if and only if $V \setminus C$ is a stable set. The same is true for the $K_i$-cover and $K_i$-free problems: a subset $C \subset \KK_{i-1}$ is a $K_i$-cover if and only if $\KK_{i-1} \setminus C$ is $K_i$-free. Therefore, for a given set of weights on $\KK_{i-1}$, optimal solutions to the two problems are complementary, and so solving one solves the other.

In this paper, we consider the polytope associated with the $K_i$-free problem. Let 
$P_i(G) = \hbox{conv}(\{\chi_S: S \subset \KK_{i-1}(G) \text{ and $S$ is $K_i$-free}\})$, 
the convex hull of the incidence vectors of the $K_i$-free sets. Note that $P_i(G) \subseteq [0,1]^{\KK_{i-1}(G)}$.

As the $K_i$-free problem is NP-complete (see \cite{conforti}), we cannot expect a small description of $P_i(G)$ for general graphs $G$. However, for certain classes of facets of $P_i(G)$, Conforti, Corneil, and Mahjoub \cite{conforti} show that we can solve the separation problem in polynomial time, allowing us to optimize efficiently over a relaxation of $P_i(G)$. We provide a strictly tighter relaxation of $P_i(G)$, improving their optimization result, but without proving the existence of polynomial separation oracles for any new family of facets.

The structure of this paper is: in Section 2, we outline the main algebraic
machinery, {\em theta bodies}, a semidefinite relaxation hierarchy. In Section 3
we show that the $K_i$-$p$-hole facets are valid on $\lceil i/2 \rceil$ level of the theta body hierarchy. Finally,
in Section 4 we focus on the triangle free problem. We show that in the case of
$G = K_n$, the theta body relaxations cannot converge in less than $(n-2)/4$ steps.
We also use a result of Krivelevich \cite{krivelevich}  to show an integrality gap of 2 for the second theta body.

\subsection*{Acknowledgements}
The authors would like to thank the anonymous referees for their helpful comments.
The authors were partially supported on this project
as follows: JG by `Centro de Matem\'{a}tica da Universidade de Coimbra' and `Funda\c{c}\~{a}o para a Ci\^{e}ncia e a Tecnologia', through European program COMPETE/FEDER; and JP by NSF grant DMS-1115293.

\section{Theta bodies}
Theta bodies are semidefinite approximations to the convex hull of an algebraic variety. For background, see \cite{frg} and \cite{gpt}. Here we state the necessary results for this paper without proofs.

Let $V \subseteq \RR^n$ be a finite point set. One description of the convex hull of $V$ is as the intersection of all affine half spaces containing $V$ (recall that $f|_V$ is the restriction of $f$ to $V$):
$$\hbox{conv}(V) = \{x \in \RR^n: f(x) \ge 0 \hbox{ for all linear $f$ such that } f|_V \ge 0\}.$$
Since it is computationally intractable to find whether $f|_V \ge 0$, we relax this condition. Let $I$ be the vanishing ideal of $V$, i.e., the set of all polynomials vanishing on $V$. Recall that $f \equiv g \mod I$ means $f - g \in I$, and implies that $f(x) = g(x)$ for all $x \in V$. A function $f$ is said to be a sum of squares of degree at most $k$ mod $I$, or {\em $k$-sos mod $I$}, if there exist functions $g_j$, $j=1,\ldots,m$ with degree at most $k$, such that $f \equiv \sum_{j=1}^mg_j^2$ mod $I$. If $f$ is $k$-sos mod $I$ for any $k$, it is clear that $f|_V \ge 0$ since $g_j^2$ is visibly nonnegative on $V$. Therefore, we make the following definition of $\THETABODY_k(I)$, the $k$-th theta body of $I$:
$$\THETABODY_k(I) = \{x \in \RR^n: f(x) \ge 0 \hbox{ for all linear $f \equiv$ $k$-sos mod $I$}\}.$$
The reason why the theta bodies $\THETABODY_k(I)$ provide a computationally tractable relaxation of $\hbox{conv}(V)$ is that the membership problem for $\THETABODY_k(I)$ can be expressed as a semidefinite program, using {\em moment matrices} that are reduced mod $I$. 

For what follows, we will restrict ourselves to a special class of varieties, and suppose that our variety $V \subseteq \{0,1\}^n$ and is {\em lower-comprehensive}; i.e., if $x\le y$ componentwise, and $y \in V$, then $x \in V$. Additionally, we will always assume that $V$ contains the canonical
basis of $\RR^n$, $\{e_1, \cdots, e_n\}$, as otherwise we could restrict ourselves to a subspace. All combinatorial optimization problems of avoiding certain finite list of configurations, such as stable set, $K_i$-free, etc., have lower-comprehensive varieties. The restriction to this class is not necessary, but makes the theta body exposition simpler. In particular, the ideal of a lower-comprehensive variety has the following simple description.

\begin{lemma} \label{ideal}
Let $V$ be a lower-comprehensive subset of $\{0,1\}^n$. Then its vanishing ideal is given by
$$I=\langle x_j^2 - x_j: j = 1, \ldots, n;
x^S: S \notin V\rangle,$$ and a basis for $\RR[V] = \RR[x]/I$ is given by $B = \{x^S: S \in V\},$
where $x^S := \prod_{i \in S} x_i$ is a shorthand used throughout the paper.
\end{lemma}

Another important fact about $\THETABODY_k(I)$ in this setting (when $I$ is a real ideal) is that a linear inequality $f(x) \geq 0$ is valid on $\THETABODY_k(I)$ if and only if $f$ is actually $k$-sos modulo $I$.
In Section 3, we will prove that certain facet-defining inequalities of $P_i(G)$ are also valid on its theta relaxations $\THETABODY_k(I)$ by presenting a sum of squares representation modulo the ideal. For now, we observe that by considering degrees, we can get a bound on which theta bodies are trivial; that is, equal to the hypercube $[0,1]^n$.

\begin{lemma} \label{lowerbound}
Let $V \subseteq \{0,1\}^n$ be lower-comprehensive, and suppose that all elements $x
\notin V$ have $\sum_j x_j \ge k$. Let $I$ be the vanishing ideal of $V$. Then 
for $l < k/2$, $\THETABODY_l(I) = [0,1]^n$.
\end{lemma}
\begin{proof}
Let $f$ be linear with $f \equiv \sum_j g_j^2 \mod I$ with each $g_j$ of degree
at most $l$. Then $f - \sum_j g_j^2 =: F \in I$, and $F$ has degree at most $2l
< k$, since $\deg(g_j^2) \le 2l$. But the basis from Lemma \ref{ideal} is a Groebner basis, and the only
elements with degree less than $k$ are $x_j^2 - x_j$, so $F \in I' := \langle x_j^2 - x_j ;j = 1, \ldots, n\rangle$. Thus $\THETABODY_l(I) \supseteq \THETABODY_l(I') = [0,1]^n$.
\end{proof}

Let $V_k$ be the subset of $V$ whose elements have at most $k$ entries equal to one. For convenience, we will often identify the elements of $V$, characteristic vectors $\chi_S$ for $S \subseteq \{1,\ldots,n\}$, with their supports, via $S \leftrightarrow \chi_S$. Given $y \in \RR^{V_{2k}}$ we denote the {\em reduced 
moment matrix} of $y$ with respect to $I$ to be the matrix $M_{V_k}(y) \in \RR^{V_k \times V_k}$ defined by 
$$[M_{V_k}(y)]_{X,Y}=
\left\{  
\begin{array}{ll} 
y_{X \cup Y} & \ \ \textrm{ if } X \cup Y \in V, \\
\\
0            & \ \ \textrm{ otherwise. }   
\end{array} 
\right.$$

With these matrices we can finally give a semidefinite description of $\THETABODY_k(I)$.

\begin{proposition}
With $I$ and $V$ as before, $\THETABODY_k(I)$ is the canonical projection onto $\mathbb{R}^n$ via the coordinates $(y_{e_1}, \cdots, y_{e_n})$ of the set
$$\{y \in \RR^{V_{2k}}\ : \ M_{V_k}(y) \succeq 0 \textrm{ and } y_{0}=1\}.$$
In particular, optimizing to arbitrary fixed precision over $\THETABODY_k(I)$
can be done in time polynomial in $n$, for fixed $k$.
\end{proposition}

Now we can consider the specific case of the $K_i$-free problem. Here the variety $V \subseteq \RR^{\KK_{i-1}(G)}$ is the set of characteristic vectors of $K_i$-free subsets of $\KK_{i-1}(G)$, $V_k$ is the subset of $V$ of elements of size at most $k$, and $I$ is the vanishing ideal of $V$, described by Lemma \ref{ideal}. Since the $K_i$s in $G$ are the minimal elements not in $V$, by Lemma \ref{ideal} we can write the ideal $I$ as follows.
$$I=\langle x_j^2 - x_j :j \in \KK_{i-1}(G); \prod_{j \subseteq K}x_j: K \in \KK_i(G)\rangle.$$

For example, let $G$ be a triangle, with edges A, B, C, and consider the triangle free problem on $G$. Then the ideal is
$$I = \langle x_A^2 - x_A, x_B^2 - x_B,  x_C^2 - x_C, x_Ax_Bx_C \rangle,$$
and the variety $V$ is as follows.
$$V = \{\emptyset,\{A\},\{B\},\{C\},\{A,B\},\{A,C\},\{B,C\}\} \equiv \{0,1,2,3,4,5,6\}.$$ 
Note that here, we again use our identification of sets with their characteristic vectors. To avoid writing, e.g., $y_{\{A,C\}}$ or even $y_{\chi_{\{A,C\}}}$, we label the elements of $V$ by numbers as above.
Then the moment matrix $M_{V_2}(y)$ is as follows:
$$M_{V_2}(y) = 
\left[
\begin{array}{ccccccc}
y_0 & y_1 & y_2 & y_3 & y_4 & y_5 & y_6 \\
y_1 & y_1 & y_4 & y_5 & y_4 & y_5 & 0 \\
y_2 & y_4 & y_2 & y_6 & y_4 & 0 & y_6 \\
y_3 & y_5 & y_6 & y_3 & 0 & y_5 & y_6 \\
y_4 & y_4 & y_4 & 0 & y_4 &0 & 0 \\
y_5& y_5 & 0 & y_5 & 0 & y_5 & 0 \\
y_6 & 0 & y_6 & y_6 & 0 & 0 & y_6 \\
\end{array}
\right]$$
Projecting the set $\{y: y_0 = 1, M_{V_2}(y) \succeq 0 \}$ onto $(y_1,y_2,y_3)$ gives $\THETABODY_2(I)$ for this graph.

\section{Polynomial-time algorithm}
A graph $H$ is a $K_i$-$p$-hole if $H$ is the union of $G_1, \ldots, G_p$, each a copy of $K_i$, where $G_j$ and $G_l$ share a common $K_{i-1}$ if and only if $j-l = \pm 1 \mod p$; see Figure \ref{9holeintro}. Theorem 3.5 in \cite{conforti} establishes that for $i \ge 3$ and odd $p$, the inequality $\sum_{\KK_{i-1}(H)} x_j \le (\frac{p-1}{2})(2i-3)+i-2$ defines a facet of $P_i(G)$ for each induced $K_i$-$p$-hole $H$ of $G$.
We will show that the facets corresponding to induced $K_i$-$p$-holes are valid
on $\THETABODY_{\lceil i/2 \rceil}(I)$, which can be optimized over in
polynomial time for fixed $i$, and relate this complexity result with the ones in Conforti, Corneil and Mahjoub \cite{conforti}. Note that in this section, the ideal $I$ always refers to the $K_i$-free problem, and the associated graph $G$ will be clear from context.
Therefore, we will say $k$-sos for $k$-sos mod $I$.

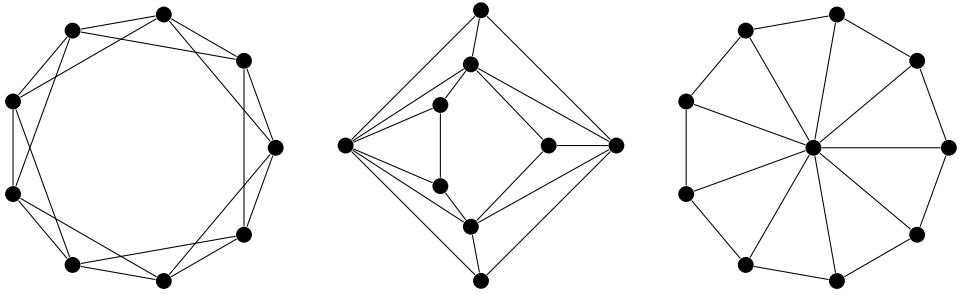
\begin{figure}[htd]
\begin{tikzpicture}
  [scale=1.8,every node/.style={circle,fill=black,inner
sep=0pt, minimum width=6pt}]  \node (0) at
(1.00000000000000,0.000000000000000) {};
  \node (1) at (0.766044443118978,0.642787609686539) {};
  \node (2) at (0.173648177666930,0.984807753012208) {};
  \node (3) at (-0.500000000000000,0.866025403784439) {};
  \node (4) at (-0.939692620785908,0.342020143325669) {};
  \node (5) at (-0.939692620785908,-0.342020143325669) {};
  \node (6) at (-0.500000000000000,-0.866025403784439) {};
  \node (7) at (0.173648177666930,-0.984807753012208) {};
  \node (8) at (0.766044443118978,-0.642787609686540) {};

  \draw (0) -- (1);
  \draw (0) -- (2);
  \draw (0) -- (7);
  \draw (0) -- (8);
  \draw (1) -- (2);
  \draw (1) -- (3);
  \draw (1) -- (8);
  \draw (2) -- (3);
  \draw (2) -- (4);
  \draw (3) -- (4);
  \draw (3) -- (5);
  \draw (4) -- (5);
  \draw (4) -- (6);
  \draw (5) -- (6);
  \draw (5) -- (7);
  \draw (6) -- (7);
  \draw (6) -- (8);
  \draw (7) -- (8);

\end{tikzpicture}\hfill
\begin{tikzpicture}
  [scale=.9,every node/.style={circle,fill=black,inner
sep=0pt, minimum width=6pt}]  \node (0) at (1,0) {};
  \node (1) at (-0.150000000000000,1.20000000000000) {};
  \node (2) at (-0.600000000000000,0.600000000000000) {};
  \node (3) at (-0.600000000000000,-0.600000000000000) {};
  \node (4) at (-0.150000000000000,-1.20000000000000) {};
  \node (5) at (2,0) {};
  \node (6) at (0,2) {};
  \node (7) at (-2,0) {};
  \node (8) at (0,-2) {};

  \draw (0) -- (1);
  \draw (0) -- (4);
  \draw (0) -- (5);
  \draw (1) -- (2);
  \draw (1) -- (5);
  \draw (1) -- (6);
  \draw (1) -- (7);
  \draw (2) -- (3);
  \draw (2) -- (7);
  \draw (3) -- (4);
  \draw (3) -- (7);
  \draw (4) -- (5);
  \draw (4) -- (7);
  \draw (4) -- (8);
  \draw (5) -- (6);
  \draw (5) -- (8);
  \draw (6) -- (7);
  \draw (7) -- (8);

\end{tikzpicture}\hfill
\begin{tikzpicture}
  [scale=1.8,every node/.style={circle,fill=black,inner
sep=0pt, minimum width=6pt}]  \node (0) at
(1.00000000000000,0.000000000000000) {};
  \node (1) at (0.766044443118978,0.642787609686539) {};
  \node (2) at (0.173648177666930,0.984807753012208) {};
  \node (3) at (-0.500000000000000,0.866025403784439) {};
  \node (4) at (-0.939692620785908,0.342020143325669) {};
  \node (5) at (-0.939692620785908,-0.342020143325669) {};
  \node (6) at (-0.500000000000000,-0.866025403784439) {};
  \node (7) at (0.173648177666930,-0.984807753012208) {};
  \node (8) at (0.766044443118978,-0.642787609686540) {};
  \node (9) at (0.000000000000000,0.000000000000000) {};

  \draw (0) -- (1);
  \draw (0) -- (8);
  \draw (0) -- (9);
  \draw (1) -- (2);
  \draw (1) -- (9);
  \draw (2) -- (3);
  \draw (2) -- (9);
  \draw (3) -- (4);
  \draw (3) -- (9);
  \draw (4) -- (5);
  \draw (4) -- (9);
  \draw (5) -- (6);
  \draw (5) -- (9);
  \draw (6) -- (7);
  \draw (6) -- (9);
  \draw (7) -- (8);
  \draw (7) -- (9);
  \draw (8) -- (9);

\end{tikzpicture}
	\caption{Three non-isomorphic $K_3$-9-holes.}
	\label{9holeintro}
\end{figure}

The first lemma is an auxiliary result that certain functions are sums of squares. For an ideal $I$, a function $f$ is said to be {\em idempotent} mod $I$ if $f^2 \equiv f \mod I$. Since an idempotent is visibly a square, we can use it as a summand in our sum of squares. In practice, idempotents end up being very useful in constructing sums of squares.

\begin{lemma}\label{ki}\
Suppose 
$A \subseteq B \subseteq \KK_{i-1}(K_i)$. Denote the variables corresponding to elements of $\KK_{i-1}(K_i)$ by $\{x_k:1 \le k \le i\}$.
Then $f(x) = |B \setminus A| - x^A+x^B- \sum_{k\in B\setminus A} x_k $ is $|B|$-sos mod $I$.
\end{lemma}
\begin{proof}
Let $A = A_1 \subset A_2 \ldots \subset A_m = B$, where $|A_{k+1} \setminus A_k|
= 1$, and without loss of generality, the variable corresponding to the element of $A_{k+1} \setminus A_k$ is $x_k$. Check that $g_k(x) = 1-x_k-x^{A_k} + x^{A_{k+1}}$ is idempotent mod $I$. Adding them
up we get that $f(x) = \sum_{k=1}^{m-1} g_k(x)$. Since each summand has degree at most $|B|$ the assertion holds.
\end{proof}

The stable set polytope $\textup{STAB}(G)$ has a fractional relaxation $\textup{FRAC}(G)$, given by imposing nonnegativities $x_i \ge 0$, and inequalities $x_i + x_j \le 1$ for each edge $(i,j)$ of $G$. Similarly, we can define a fractional $K_i$-free polytope $\textup{FRAC}_i(G)$ by imposing nonnegativities, and the inequalities $\sum_{k \in \KK_{i-1}(H)}x_k \le i-1$ for each $H \in \KK_i(G)$. The following corollary shows that these inequalities are $\lceil i/2 \rceil$-sos, and therefore that the relaxation $\THETABODY_{\lceil i/2 \rceil}(I) \subseteq \textup{FRAC}_i(G)$. This is parallel to the result that the Lov\'{a}sz theta body lies inside $\textup{FRAC}(G)$.

\begin{corollary} \label{frac}
The inequality $\sum_{k \in \KK_{i-1}(H)}x_k \le i-1$ is valid on $\THETABODY_{\lceil i/2 \rceil}(I)$ for every $H \in \KK_i(G)$.
\end{corollary}
\begin{proof}
Let $J$ be a subset of $\KK_{i-1}(H)$ of size $\lceil i/2 \rceil$. Applying Lemma \ref{ki} with 
$A=\emptyset$ and $B=J$ we see that 
 $$ f(x) =  |J|-1 + x^J - \sum_{l \in J} x_l $$ is $|J|$-sos. Similarly 
$$g(x)= |J^c|-1 + x^{J^c} - \sum_{l \in J^c} x_l $$ is $|J^c|$-sos ($J^c$ is the complement of $J$ in $\KK_{i-1}(H)$). Finally observe that 
$h(x) = 1 - x^J - x^{J^c}$ is idempotent. Since these polynomials are all $\lceil i/2 \rceil$-sos, it remains to observe that their sum,
$$f(x)+g(x)+h(x)=i-1-\sum_{k \in \KK_{i-1}(H)}x_k,$$
is also $\lceil i/2 \rceil$-sos.
\end{proof}

Now we are ready to prove that the $K_i$-$p$-hole inequalities are valid on $\THETABODY_{\lceil i/2 \rceil}(I)$. Recall that if $H$ is a $K_i$-$p$-hole, we write $G_1,\ldots,G_p$ for the $K_i$s in $H$, with adjacent $K_i$ sharing a common $K_{i-1}$. 

\begin{theorem} \cite[Theorem 3.4]{conforti}
If $G$ has an induced $K_i$-$p$-hole $H$, then the inequality
$$\frac{p-1}{2} (2i-3)+i-2 - \sum_{i \in H} x_i \geq 0$$
defines a facet of $P_i(G)$ for $i \ge 3$.
\end{theorem}

For even $p$ these  inequalities are still valid, but not facets anymore. In the next result we give a sums of squares certificate to
the validity of these inequalities.

\begin{lemma}\label{kiphole}
The $K_i$-$p$-hole inequalities are $\lceil i/2 \rceil$-sos for $p$ odd.
\end{lemma}
\begin{proof}

Let $p = 2k+1$. For each $l=1,\ldots,2k+1$, there is exactly one $K_{i-1}$
common to $G_l$ and $G_{l-1}$ (taking indices mod $2k+1$). Denote the
corresponding variable by $x_l$. Now fix $l$. Let the variables $\{y_k\}$
correspond to the $K_{i-1}$ contained in only one $G_l$. Then the variables corresponding to $\KK_{i-1}(G_l)$ are $\{x_l,x_{l+1},y_1,\ldots,y_{i-2}\}$. We will show that $p_l(x,y) = i-2 - \sum y_k - x_lx_{l+1}$ is $\lceil i/2 \rceil$-sos.

Let $J_1 = \{1,\ldots,\lceil i/2 \rceil - 2\}$ and $J_2 = \{\lceil i/2 \rceil -
1, \ldots,i-2\}$. Applying Lemma \ref{ki}, we see that the following two functions are $\lceil i/2 \rceil$-sos. First apply the lemma with $A = \{x_l,x_{l+1}\}$ and $B = \{y_j: j \in J_1\} \cup \{x_l,x_{l+1}\}$:
$$f(x,y)=|J_1| - x_lx_{l+1} + x_lx_{l+1}y^{J_1}- \sum_{j \in J_1} y_j .$$
Second, take $A = \emptyset$ and $B = J_2$:
$$g(x,y)=|J_2| - 1  + y^{J_2}- \sum_{j \in J_2} y_j .$$
Finally, observe that the following is idempotent mod $I$:
$$h(x,y) = 1-x_lx_{l+1}y^{J_1} - y^{J_2}.$$
Adding these up we get that $p_l(x,y) = f(x,y)+g(x,y)+h(x,y)$ is $\lceil i/2 \rceil$-sos. Now with $p(x,y) = \sum_{l=1}^{2k+1} p_l(x,y)$, we have that $p$ is $\lceil i/2 \rceil$-sos:
$$p(x,y) = (2k+1)(i-2) - \sum_{l=1}^{2k+1}\sum_{y_k \subseteq G_l}y_k - \sum_{l=1}^{2k+1}x_lx_{l+1},$$
where the sum $\sum y_k$ is over all $K_{i-1}$ contained in a unique $K_i$. It
remains to show that $k-\sum x_l + \sum x_lx_{l+1}$ is $\lceil i/2 \rceil$-sos.
Observe that this is attained by adding the following two quantities, each of
which is a sum of idempotents mod $I$:
$$ \sum_{l=1}^k \left(1 - x_{2l-1} - x_{2l} - x_{2l+1} + x_{2l-1}x_{2l} +
x_{2l-1}x_{2l+1} + x_{2l}x_{2l+1}\right),$$
$$ \sum_{l=2}^k (x_{2l-1} - x_{2l-1}x_{1} - x_{2l-1}x_{2l+1} + x_{2l+1}x_{1}).$$
\end{proof}

In view of Lemma \ref{kiphole}, we see that the $K_i$-$p$-hole inequalities are
valid on $\THETABODY_{\lceil i/2 \rceil}(I)$. But since these inequalities define
facets of $P_i(G) \subseteq \THETABODY_{\lceil i/2 \rceil}(I)$, we see that they also
define facets of $\THETABODY_{\lceil i/2 \rceil}(I)$.

In Section 3.3 of \cite{conforti}, Conforti, Corneil, and Mahjoub show that polynomial time separation oracles exist for the following facets of $P_i(G)$.
\begin{enumerate}
\item Nonnegativities $0 \le x_k \le 1$,
\item $K_i$ (clique) inequalities $\sum_{k \subset K_i}x_k \le i-1$,
\item Odd wheels of order $i-1$.
\end{enumerate}
Define the polytope $Q(G)$ as the intersection of these facets, and define $Q'(G)$ by replacing (3) by 
\begin{enumerate}[label=(\arabic*$'$),ref=(\arabic*$'$),start=3]
\item $K_i$-$p$-hole facets,
\end{enumerate}
a superclass of (3).

The separation oracle provided by Conforti, Corneil, and Mahjoub \cite{conforti} allows us to use the ellipsoid method to optimize over $Q(G)$ in polynomial time. However, it is stated as an open problem in \cite{conforti} whether there is also such a polynomial time oracle for the class ($3'$), which would allow us to optimize in polynomial time over $Q'(G)$.

The results in this section show that $\THETABODY_{\lceil i/2 \rceil}(I)$, over which we can optimize in polynomial time (for
fixed $i$ and to fixed arbitrary precision), is a tighter relaxation than $Q'(G)$. Precisely, we have the following inclusions:
$$P_{i}(G) \subseteq  \THETABODY_{\lceil i/2 \rceil}(I) \subseteq Q'(G) \subseteq Q(G).$$
A big advantage of this result is how easy it is to use in practice. The polynomial time results derived from the separation oracle rely on the ellipsoid method, which is numerically unstable and poor in practice even for small instances. By contrast, optimizing over the theta body is a standard semidefinite program; hence it can be done using interior point methods and can be straightforwardly implemented using any off-the-shelf solver. The original question in \cite{conforti} is however still open, since we have not provided any separation oracle for the class ($3'$).

There are other families of facets of $P_i(G)$ for which efficient separation oracles are given in \cite{conforti}. We have not treated them here, as they would not yield any new polynomial time results.

\section{Related Problems}
Here we apply two results appearing in the literature to the triangle free problem.

\subsection{A lower bound on theta convergence}
In Section 3, we showed that the earliest possible theta body, $\THETABODY_{\lceil i/2 \rceil}(I_G)$, satisfies several inequalities defining facets of $P_i(G)$. However, in general it can take many steps of the theta hierarchy before a given facet of $P_i(G)$ is valid on $\THETABODY_k(G)$. This is the case even for the triangle free problem. In particular, we will show:
\begin{theorem}\label{trianglecut}
For $k < \frac{n-2}{4}$, $P_3(K_n) \subsetneq \THETABODY_{k}(I_{K_n})$.
\end{theorem}

To prove Theorem 4.1, we will apply a result of Laurent \cite{moniquestuff} on
the {\em cut polytope}. Let $G=(N,E)$ be a graph. A {\em cut} in $G$ arises from
a partition of the node set $N$ into two sets $S_1$ and $S_2$, whereupon the
associated cut is the set of edges from $S_1$ to $S_2$. Let $C_G \subseteq
\{0,1\}^E$ be the collection of characteristic vectors of cuts in $G$. Then
$\textup{CUT}(G) = \textup{conv}(C_G)$ is the cut polytope of $G$. Similarly,
define $T_G \subseteq \{0,1\}^E$ to be the set of characteristic vectors of
triangle free sets in $G$, and as before, $P_3(G) = \textup{conv}(T_G)$. Note
that a cut is by definition a bipartite graph; hence, it is triangle free. Therefore $C_G \subseteq T_G$ and $\textup{CUT}(G) \subseteq P_3(G)$. The theta body approach has also been applied to the cut polytope by Gouveia, Laurent, Parrilo, and Thomas \cite{GLPT}. The following lemma shows that inclusion among varieties extends to inclusion of theta bodies.

\begin{lemma}\label{ideal_inclusion}
Let $X \subseteq Y$ be two real varieties, with ideals $I(X)$ and $I(Y)$. Then for any $k$, $\THETABODY_k(I(X)) \subseteq \THETABODY_k(I(Y))$.
\end{lemma}
\begin{proof}
If $X \subseteq Y$, then the reverse inclusion holds for their ideals: $I(Y) \subseteq I(X)$. Any function which is $k$-sos mod $I(Y)$ is then also $k$-sos mod $I(X)$. The result follows from the definition of $\THETABODY_k(I)$.
\end{proof}
In particular, since $C_G \subseteq T_G$, we have $\THETABODY_k(I(C_G))
\subseteq \THETABODY_k(I(T_G))$ for all $k$. Note that $I(T_G) = I_G$ in our
notation from the $K_i$-free problem.

For the complete graph $K_n$, when $n$ is odd, the inequality
\begin{equation} \sum_{e \in E} x_e \le \frac{n^2-1}{4}\end{equation}
defines a facet of both $P_3(K_n)$ and $\textup{CUT}(K_n)$.
\begin{theorem}{\cite[Theorem 6]{moniquestuff}}
For $k < \frac{n-2}{4}$, $\textup{CUT}(K_n) \subsetneq
\THETABODY_{k}(I(C_{K_n}))$. In particular, equation (1) does not hold on $\THETABODY_{k}(I(C_{K_n}))$.
\end{theorem}
In \cite{moniquestuff} this result appears in terms of a different but related relaxation.
A translation of that result
to the theorem above can be found in Example 3.9 of \cite{GLPT}. We can now prove Theorem 4.1.

\begin{proof}
By Theorem 4.3, there is a point $x \in \THETABODY_{k}(I(C_{K_n}))$ violating (1). But by Lemma 4.2, $x \in \THETABODY_{k}(I(T_{K_n}))$. Since (1) is valid on $P_3(K_n)$, $x \notin P_3(K_n)$.
\end{proof}

This implies that the theta body hierarchy does not polynomially capture the $K_n$ inequalities, as the size of the reduced moment matrices associated with the $\lceil\frac{n-2}{4}\rceil$-th theta body is exponential in $n$. It is still an open question, for both $\textup{CUT}(K_n)$ and $P_3(K_n)$, what is the
smallest $k$ so that the $k$-th theta body is exact.

Recall that for $i=2$, the $K_i$-free problem is simply the stable set problem. In that case it is well known that the clique inequalities
are valid for the first theta body relaxation. A simple byproduct of Theorem \ref{trianglecut} is that this fact fails to generalize even to $i=3$, as it is impossible to capture all the clique inequalities with a constant rank of theta bodies in this case.

\subsection{An integrality gap for triangle cover}
Let $G$ be a graph. A {\it triangle cover} is a collection of edges in $G$, containing at least one edge from every triangle in $G$. Let $\tau(G)$ be the minimum-size triangle cover in $G$ (in the language of the introduction, the $K_3$-cover problem with unit weights). Let $I$ be the ideal of the triangle cover problem. Define the following semidefinite relaxation:
$$\tau^\dagger(G) = \min \left\{\sum_{e \in E} x_e: x \in \THETABODY_2(I) \right\}.$$
Note that since $C$ is a triangle cover if and only if $E \setminus C$ is a triangle free set, we can restate any statements about theta bodies for the triangle free problem using the change of variables $x \mapsto 1-x$.

We can also define a natural LP relaxation for the triangle cover problem. Let
$$\tau^*(G) = \min \left\{\sum_{e \in E} x_e: x \in [0,1]^E \hbox{ and for all triangles $\Delta$,} \sum_{e \in \Delta} x_e \ge 1\right\}.$$
Krivelevich \cite{krivelevich} proved that $\tau(G) \le 2\tau^*(G)$. We can apply this to prove an integrality gap for $\tau^\dagger(G)$.
\begin{theorem}
For any graph $G$, $\tau^\dagger(G) \ge \frac{\tau(G)}{2}$.
\end{theorem}
\begin{proof}
By Corollary 3.2, $\tau^\dagger(G) \ge \tau^*(G)$, as the inequalities defining $\tau^*(G)$ are valid on the second theta body. Combining this with Krivelevich's inequality gives the result.
\end{proof}

Another way to interpret Krivelevich's result is in terms of a conjecture of Tuza. Define a {\it triangle packing} in a graph $G$ to be a collection of triangles in $G$, no two of which share an edge. Let $v(G)$ be the maximum-size triangle packing in $G$. It is an easy exercise to check that $v(G) \le \tau(G) \le 3v(G)$. However, Tuza conjectured in \cite{tuza} that the stronger inequality $\tau(G) \le 2v(G)$ holds for all graphs $G$. The problem is currently open; see \cite{haxell} for more information. $v(G)$ also has a natural LP relaxation.
$$v^*(G) = \max \left\{\sum_{\Delta \in T} y_\Delta: y \in [0,1]^T \hbox{ and for all edges $e$,} \sum_{e \in \Delta} y_\Delta \le 1\right\}$$
By LP duality, $\tau^*(G) = v^*(G)$. Krivelevich \cite{krivelevich} also proved that  $v^*(G) \le 2v(G)$. After applying the duality $\tau^*(G) = v^*(G)$, these become fractional versions of Tuza's conjecture: $\tau(G) \le 2v^*(G)$ and $\tau^*(G) \le 2v(G)$. A natural question to ask is whether, given the SDP relaxation $\tau^\dagger(G)$, whether the ``semidefinite version'' of Tuza's conjecture would hold: $\tau^\dagger(G) \le 2v(G)$.

\bibliographystyle{plain}
\bibliography{KiCover}
\end{document}